\documentclass[letterpaper,11pt]{amsart}
\usepackage{graphicx, amscd, amssymb}
\newtheorem{lem}{Lemma}[section]
\newtheorem{thm}[lem]{Theorem}
\newtheorem{pro}[lem]{Proposition}

\newtheorem{exa}[lem]{Example}
\newtheorem{con}[lem]{Conjecture}

\newcommand{\ZZ}{{\mathbb{Z}}}

\newcommand{\A}{{\mathcal{A}}}
\newcommand{\B}{{\mathcal{B}}}
\newcommand{\C}{{\mathcal{C}}}

\newcommand{\U}{{\textsf{U}}}

\newcommand{\D}{{\textsf{D}}}

\renewcommand{\P}{{\mathcal{P}}}

\newcommand{\T}{{\mathcal{T}}}

\begin{document}

\title{Exterior Pairs and Up Step Statistics on Dyck Paths}

\author{Sen-Peng Eu}
\address{Department of
Applied Mathematics, National University of Kaohsiung, Kaohsiung
811, Taiwan, ROC} \email{speu@nuk.edu.tw}

\author{Tung-Shan Fu}
\address{Mathematics
Faculty, National Pingtung Institute of Commerce, Pingtung 900,
Taiwan, ROC} \email{tsfu@npic.edu.tw}

\thanks{Partially supported by  National Science Council, Taiwan under
grants 98-2115-M-390-002 (S.-P. Eu) and 97-2115-M-251-001 (T.-S. Fu).}



\maketitle

\begin{abstract}
Let $\C_n$ be the set of Dyck paths of length $n$. In this paper, by a new automorphism
of ordered trees, we prove that the statistic
`number of exterior pairs', introduced by A. Denise and R. Simion,
on the set $\C_n$ is equidistributed with the statistic
`number of up steps at height $h$ with $h\equiv 0$ (mod 3)'. Moreover, for $m\ge 3$,
we prove that the two statistics
`number of up steps at height $h$ with $h\equiv 0$ (mod $m$)'
and `number of up steps at height $h$ with $h\equiv m-1$ (mod $m$)'
on the set $\C_n$ are `almost equidistributed'.
Both results are proved combinatorially.
\end{abstract}

\section{Introduction}

Let $\C_n$ denote the set of lattice paths, called \emph{Dyck paths} of length $n$, in the plane $\ZZ\times \ZZ$
from the origin to the point $(2n,0)$ using \emph{up step} $(1,1)$ and \emph{down step} $(1,-1)$ that never pass below the $x$-axis.
Let $\U$ and $\D$ denote an up step and a down step, respectively.
In \cite{DS}, Denise and Simion introduced and investigated the two statistics
`pyramid weight' and `number of exterior pairs' on the set $\C_n$.
A \emph{pyramid} in a Dyck path is a section of the form $\U^h\D^h$, a succession of $h$ up steps followed immediately by
$h$ down steps, where $h$ is called the {\em height} of the pyramid. The pyramid is {\em maximal} if it is not contained in a higher pyramid. The {\em pyramid weight} of a Dyck path is the sum of the heights of its maximal pyramids.  An \emph{exterior pair} in a Dyck path is a pair consisting of
an up step and its matching down step which do not belong to any pyramid.
For example, the path shown in Figure \ref{fig:pyramid} contains three maximal
pyramids with a total weight of 4 and two exterior pairs.

\begin{figure}[ht]
\begin{center}
\includegraphics[width=1.6in]{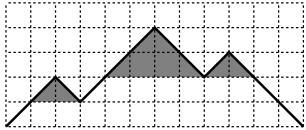}
\end{center}
\caption{\small A Dyck path with three maximal pyramids and two exterior pairs.}
\label{fig:pyramid}
\end{figure}

Since a Dyck path in $\C_n$ with a pyramid weight of $k$ contains $n-k$ exterior pairs, both of the statistics are essentially equidistributed on the set $\C_n$. However, they seem to be `isolated' from other statistics in the sense that so far
there are no known statistics that share the same distribution with them.
In the first part of this work, we discover one and establish an explicit connection with the statistic `number of exterior pairs'.

For a Dyck path, an up step that rises from the line $y=h-1$ to the line $y=h$
is said to be at \emph{height} $h$. It is well known \cite{Krew} that the number of paths in $\C_n$ with $k$ up steps at even height is enumerated by the \emph{Narayana number}
\[ N_{n,k}=\frac{1}{n}{{n}\choose{k}}{{n}\choose{k+1}},
\]
for $0\le k\le n-1$. Note that $\sum_{k=0}^{n-1}N_{n,k}=\frac{1}{n+1}{{2n}\choose{n}}=|\C_n|$ is the $n$th Catalan number. We consider the number $g^{(c;3)}_{n,k}$ of
the paths in $\C_n$ with $k$ up steps at height $h$ such that $h\equiv c$ (mod 3), for some $c\in\{0,1,2\}$. For example, the initial values of $g^{(c;3)}_{n,k}$ are shown in Figure \ref{fig:gnk}.

\begin{figure}[ht]
{\small
\begin{tabular}{ccc}
{$\begin{array}{c|rrrrrr} \hline
n \backslash k &   0 &   1 &   2 &  3 &  4 &  5\\
  \hline
  1   &   1 &     &     &    &   &     \\
  2   &   2 &     &     &    &   &     \\
  3   &   4 &   1 &     &    &   &     \\
  4   &   8 &   5 &   1 &    &   &     \\
  5   &  16 &  18 &   7 &  1 &   &     \\
  6   &  32 &  56 &  34 &  9 &  1&
\end{array}
$}

&

{$\begin{array}{c|rrrrrr} \hline
n \backslash k &   1 &   2 &   3 &  4 &  5 &  6\\
  \hline
  1   &   1 &     &     &    &    &  \\
  2   &   1 &   1 &     &    &    &  \\
  3   &   2 &   2 &   1 &    &    &  \\
  4   &   4 &   6 &   3 &  1 &    &  \\
  5   &   8 &  17 &  12 &  4 &  1 &  \\
  6   &  16 &  46 &  44 & 20 &  5 & 1\\
\end{array}
$}

&

{$\begin{array}{c|rrrrrr} \hline
 n \backslash k &   0 &   1 &   2 &  3 &  4  &  5\\
  \hline
  1   &   1 &     &     &    &   &     \\
  2   &   1 &   1 &     &    &   &     \\
  3   &   1 &   3 &   1 &    &   &     \\
  4   &   1 &   7 &   5 &  1 &   &     \\
  5   &   1 &  15 &  18 &  7 &  1&      \\
  6   &   1 &  31 &  56 & 34 &  9& 1
\end{array}
$}  \\ & & \\
$g^{(0;3)}_{n,k}$ & $g^{(1;3)}_{n,k}$ & $g^{(2;3)}_{n,k}$
\end{tabular}
}
\caption{\small  The distribution of Dyck paths with respect to $g^{(c;3)}_{n,k}$.}
\label{fig:gnk}
\end{figure}

To our surprise, the distribution $g^{(0;3)}_{n,k}$, shown in Figure \ref{fig:gnk},
coincides with the distribution of the statistic `number of exterior pairs'
on the set $\C_n$ (cf. \cite[Figure 2.4]{DS}).
In addition to an algebraic proof by the method of generating functions (see Example \ref{exa:GFproof}), one of the main results in this paper is a bijective proof of the equidistribution of these two statistics (Theorem \ref{thm:main-2}), which
is established by a recursive construction. To our knowledge, it is not equivalent to any previously known bijection on the set $\C_n$.

\smallskip
\begin{thm} \label{thm:main-2} For $0\le k\le n-2$, there is a bijection $\Pi:\C_n\rightarrow\C_n$ such that a path $\pi\in\C_n$ with $k$ exterior pairs is carried to the corresponding path $\Pi(\pi)$ containing $k$ up steps at height $h$ with $h\equiv 0$ (mod 3).
\end{thm}

\smallskip
Recall that a path in $\C_n$ with $k$ up steps at even height contains $n-k$ up steps at odd height and that  $N_{n,k}=N_{n,n-1-k}$ ($0\le k\le n-1$). It follows immediately that the two statistics `number of up step at even height' and `number of up steps at odd height' are equidistributed on the set $\C_n$. Specifically, the number of paths in $\C_n$ with $k$ steps at even height equals the number of paths with $k+1$ up steps at odd height. (However, the one-to-one correspondence between the two sets is not apparent.) Moreover, as one has noticed  in Figure \ref{fig:gnk} that $g^{(0;3)}_{n,k}=g^{(2;3)}_{n,k+1}$ for $k\ge 1$, the two statistics `number of up steps at height $h$ with $h\equiv 0$ (mod 3)' and 'number of up steps at height $h$ with $h\equiv 2$ (mod 3)' are almost equidistributed on the set $\C_n$.

Motivated by this fact, for an integer $m\ge 2$ and a set $R\subseteq\{0,1,\dots,m-1\}$
we study the enumeration of the paths in $\C_n$ with $k$ up steps at height $h$ such
that $h\equiv c$ (mod $m$) and $c\in R$.
Let $g^{(R;m)}_{n,k}$ denote this number and let $G^{(R;m)}$
be the generating function for $g^{(R;m)}_{n,k}$, where
\[ G^{(R;m)}=G^{(R;m)}(x,y)=\sum_{n\ge 0}\sum_{k\ge 0}  g^{(R;m)}_{n,k} y^k x^n. \]
We shall show that $G^{(R;m)}$ satisfies an equation that is expressible in terms of continued fractions (Theorem \ref{thm:continued}), which is equivalent to a
quadratic equation in $G^{(R;m)}$.
If $R$ is a singleton, say $R=\{c\}$, we write $g^{(c;m)}_{n,k}$ and $G^{(c;m)}$ instead. The other main result in this paper is to prove combinatorially that the two statistics `number of up steps at height $h$ with $h\equiv m-1$ (mod $m$)' and `number of up steps at height $h$ with $h\equiv 0$ (mod $m$)' are almost equidistributed, i.e., $g^{(0;m)}_{n,k}=g^{(m-1;m)}_{n,k+1}$, for $k\ge 1$, and $g^{(0;m)}_{n,0}=g^{(m-1;m)}_{n,0}+g^{(m-1;m)}_{n,1}$ (see Theorem \ref{thm:main}).

\smallskip
\begin{thm} \label{thm:main} For $m\ge 2$, the following equation holds.
\begin{equation} \label{eqn:vanish}
G^{(m-1;m)}-y\cdot G^{(0;m)}=\frac{(1-y)U_{m-2}(\frac{1}{2\sqrt{x}})}{\sqrt{x}U_{m-1}(\frac{1}{2\sqrt{x}})},
\end{equation}
where $U_n(x)$ denotes the $n$th Chebyshev polynomial of the second kind, $U_n(\cos\theta)=\frac{\sin((n+1)\theta)}{\sin\theta}$.
\end{thm}

We remark that $U_{m-2}(\frac{1}{2\sqrt{x}})/(\sqrt{x}U_{m-1}(\frac{1}{2\sqrt{x}}))$, a polynomial in $x$, is a generating function for the number of paths in $\C_n$ of height at most $m-2$, as pointed out by Krattenthaler \cite[Theorem 2]{Krat} (see also \cite{CW} and \cite{MV}). Note that in Eq.\,(\ref{eqn:vanish}) the terms with $y^i$ vanish, for $i\ge
2$.

\smallskip
 \section{Proof of Theorem \ref{thm:main-2}}
 In this section, we shall establish the bijection requested in Theorem \ref{thm:main-2}. A \emph{block} of a Dyck path
 is a section beginning with an up step whose starting point is on the $x$-axis and ending with the first down step that returns to the $x$-axis afterward. Dyck paths that have exactly one block are called \emph{primitive}. We remark that the requested bijection is established for primitive Dyck paths first and then for ordinary ones in a block-by-block manner. In fact, the bijection is constructed in terms of ordered trees.

 An \emph{ordered tree} is an unlabeled rooted tree where the order of the subtrees of a vertex is significant.
Let $\T_n$ denote the set of ordered trees with $n$ edges.
There is a well-known bijection $\Lambda:\C_n\rightarrow\T_n$ between Dyck paths and ordered trees \cite{DZ}, i.e., traverse the tree from the root in preorder, to each edge passed on the way down there corresponds an up step and to each edge passed on the way up there corresponds a down step. For example, Figure \ref{fig:preorder} shows a Dyck path of length 14 with 2 blocks and the corresponding ordered tree.

\begin{figure}[ht]
\begin{center}
\includegraphics[width=4.9in]{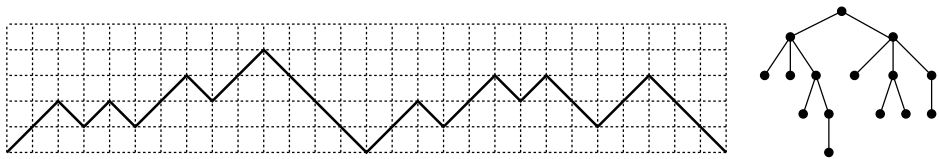}
\end{center}
\caption{\small A Dyck path and the corresponding ordered tree.} \label{fig:preorder}
\end{figure}

For an ordered tree $T$ and two vertices $u,v\in T$, we say that $v$
is a \emph{descendant} of $u$ if $u$ is contained in the path from
 the root to $v$. If also $u$ and $v$ are adjacent, then $v$ is called a \emph{child} of $u$.
A vertex with no children is called a \emph{leaf}. By a \emph{planted (ordered) tree} we mean an ordered tree whose root has only one child. (We will speak of \emph{planted trees} without including the word `ordered'.)
Let $\tau(uv)$ denote the planted subtree of
$T$ consisting of the edge $uv$ and the descendants of $v$, and let $T-\tau(uv)$
denote the remaining part of $T$ when $\tau(uv)$ is removed. In this case, the edge $uv$ is called the \emph{planting stalk} of $\tau(uv)$.
It is easy to see that the Dyck path corresponding to a planted tree is primitive.

The \emph{level} of edge $uv\in T$ is defined to be the distance from the root to the end vertex $v$. The \emph{height} of $T$ is the highest level of the edges of $T$.
The edge $uv$ is said to be \emph{exterior} if $\tau(uv)$ contains at least two leaves. One can check that the exterior edges of $T$ are in one-to-one correspondence with the exterior pairs of the corresponding Dyck path $\Lambda^{-1}(T)$. Moreover, the edges at level $h$ in $T$ are in one-to-one correspondence with the up steps at height $h$ in $\Lambda^{-1}(T)$. Hence, under the bijection $\Lambda$, the following result leads to the bijection $\Pi=\Lambda^{-1}\circ\Phi\circ\Lambda$ requested in Theorem \ref{thm:main-2}.

\smallskip
\begin{thm} \label{thm:equidistribution-trees} For $0\le k\le n-2$, there is a bijection $\Phi:\T_n\rightarrow\T_n$ such that a tree $T\in\T_n$ with $k$ exterior edges is carried to the corresponding tree $\Phi(T)$ containing $k$ edges at level $h$ with $h\equiv 0$ (mod 3).
\end{thm}

\smallskip
Our strategy is to decompose an ordered tree (from the root) into planted subtrees, find the corresponding trees of the planted subtrees, and then merge them (from their roots) together. In the following, we focus the construction of $\Phi$ on planted trees.

\smallskip
\subsection{Planted trees}
Let $\P_n\subseteq\T_n$ be the set of planted trees with $n$ edges. By a \emph{bouquet} of size $k$ ($k\ge 1$) we mean a planted tree such that there are $k-1$ edges emanating from the unique child of the root. Clearly, a bouquet is of height at most 2. Inspired by work of Deutsch and Prodinger \cite{DP}, bouquets are useful in our construction.
For convenience, the edges of a tree at level $h$ are colored \emph{red} if  $h\equiv 0$ (mod 3) and colored \emph{black} otherwise.
Now we establish a bijection $\phi:\P_n\rightarrow\P_n$ such that the exterior edges of $T\in\P_n$ are transformed to the red edges in $\phi(T)$.

\smallskip
\subsection{The map $\phi$.}

Given a  $T\in\P_n$, let $uv$ be the planting stalk of $T$. If $T$ contains no exterior edges then $T$ is a path of length $n$ and we define $\phi(T)$ to be a bouquet of size $n$. Otherwise, $T$ contains at least one exterior edge.  Note that the planting stalk $uv$ itself is one of the exterior edges of $T$. Let $w_1,\dots, w_r$ be the children of $v$, for some $r\ge 1$. Unless specified, these children are placed in numeric order of the subscripts from left to right. The tree $\phi(T)$ is recursively constructed with respect to $uv$ according to the following three cases.

\smallskip
 Case 1. \emph{Edge $vw_r$ is an exterior edge of $T$.} For $1\le j\le r$, we first construct the planted subtrees $T_j=\phi(\tau(vw_j))$. In particular, in $T_r$ we find the rightmost edge, say $xz$, at level 3. Then $\phi(T)$ is obtained from $T_r$ by adding an edge $xy$ (emanating from vertex $x$) to the right of $xz$ and adding $T_1,\dots,T_{r-1}$ under the edge $xy$ (i.e., merges the roots of $T_1,\dots,T_{r-1}$ with $y$). Note that the red edge $xy$ is created in replacement of the planting stalk $uv$ of $T$.

\smallskip
  Case 2. \emph{Edge $vw_r$ is not an exterior edge but $vw_{r-1}$ is an exterior edge.} Then $\tau(vw_r)$ is a path of a certain length, say $t$ ($t\ge 1$). For $1\le j\le r-1$, we first construct the planted subtrees $T_j=\phi(\tau(vw_j))$. In particular, let $pq$ be the planting stalk of $T_{r-1}$. Then  $\phi(T)$ is obtained from $T_{r-1}$ by adding a path $qxy$ of length 2 such that the edge $qx$ is the right most edge at level 2 (emanating from vertex $q$), and then adding $t-1$ more edges $qz_1,\dots,qz_{t-1}$ (emanating from vertex $q$) to the right of $qx$ and adding $T_1,\dots,T_{r-2}$ under the edge $xy$. Note that the planting stalk $uv$ of $T$ is replaced by the red edge $xy$ and that the subtree $\tau(vw_r)$ of $T$ is replaced by the edges $\{qx,qz_1,\dots,qz_{t-1}\}$.

 \smallskip
 Case 3. \emph{Neither $vw_{r-1}$ nor $vw_r$ is an exterior edge.} Then $\tau(vw_{r-1})$ and $\tau(vw_r)$ are paths of certain lengths. Let the lengths of $\tau(vw_{r-1})$ and $\tau(vw_r)$ be $t_1$ and $t_2$, respectively. For $1\le j\le r-2$, we first construct the planted subtrees $T_j=\phi(\tau(vw_j))$. To construct the tree $\phi(T)$, we create a path $pqxy$ of length 3, where vertex $p$ is the root. Next, add $t_1-1$ edges $qz_1,\dots,qz_{t_1-1}$ to the left of the edge $qx$ and add $t_2-1$ edges $qz_1',\dots,qz_{t_2-1}'$ to the right of the edge $qx$. Then add $T_1,\dots,T_{r-2}$  under the edge $xy$. Note that the planting stalk $uv$ of $T$ is replaced by the red edge $xy$ and that the subtree $\tau(vw_{r-1})$ (resp. $\tau(vw_{r})$) of $T$ is replaced by the edges $\{pq,qz_1,\dots,qz_{t_1-1}\}$ (resp. $\{qx,qz_1',\dots,qz_{t_2-1}'\}$).

\smallskip
\begin{exa} \label{exa:example-phi} {\rm
Let $T$ be the tree on the left of Figure \ref{fig:pqxy}. Note that the edges $uv$ and $ve$ are exterior edges. To construct $\phi(T)$, we need to form the subtrees $T_1=\phi(\tau(vc)), T_2=\phi(\tau(vd))$ and $T_3=\phi(\tau(ve))$. By Case 3 of the algorithm, $T_3$ is a path $pqxz$ of length 3, along with an edge $qh$ on the right of $qx$. Since $ve$ is an exterior edge of $T$, by Case 1, $\phi(T)$ is obtained from $T_3$ by adding the edge $xy$ and adding $T_1=yc$ and $T_2=yd$ under the edge $xy$, as shown on the right of Figure \ref{fig:pqxy}. Note that the planting stalk $uv$ of $T$ is transformed to the red edge $xy$, the rightmost one at level 3 in $\phi(T)$.
}
\end{exa}

\begin{figure}[ht]
\begin{center}
\includegraphics[width=1.75in]{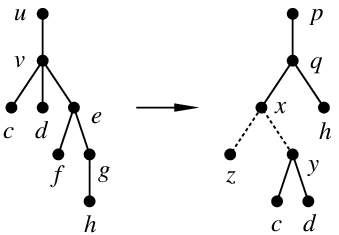}
\end{center}
\caption{\small A planted tree with and its corresponding tree.} \label{fig:pqxy}
\end{figure}

\smallskip
\begin{exa} \label{exa:example-phi-2} {\rm
Let $T$ be the tree on the left of Figure \ref{fig:pqab}. Note that the edges $uv$ and $vd$ are exterior edges. To construct $\phi(T)$, we need to form the subtrees $T_1=\phi(\tau(vc))$ and $T_2=\phi(\tau(vd))$. By Case 3 of the algorithm, $T_2$ is a path $pqab$ of length 3. Since $\tau(ve)$ is a path of length 2, by Case 2, $\phi(T)$ is obtained from $T_2$ by adding a path $qxy$ of length 2, along with the edge $qz$, and then adding $T_1=yc$ under the edge $xy$, as shown on the right of Figure \ref{fig:pqab}. Note that the planting stalk $uv$ of $T$ is transformed to the red edge $xy$, the rightmost one at level 3 in $\phi(T)$.
}
\end{exa}

\begin{figure}[ht]
\begin{center}
\includegraphics[width=1.9in]{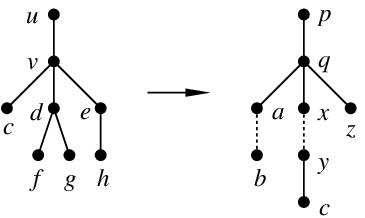}
\end{center}
\caption{\small A planted tree with and its corresponding tree.} \label{fig:pqab}
\end{figure}

\smallskip
\begin{exa} \label{exa:example-phi-3} {\rm
Let $T$ be the tree on the left of Figure \ref{fig:example-case-3}. To construct $\phi(T)$, we need to form the subtrees $T_1=\phi(\tau(vc))$ and $T_2=\phi(\tau(vd))$, which have been shown in Example \ref{exa:example-phi} and Example \ref{exa:example-phi-2}, respectively. Since neither $ve$ nor $vf$ is an exterior edge, by Case 3 of the algorithm, we create a path $pqxy$ of length 3, along with the edge $qg$ attached to the left of $qx$ and with the edges $qh, qi$ attached to the right of $qx$. As shown on the right of Figure \ref{fig:example-case-3}, the tree $\phi(T)$ is then obtained by adding $T_1=\tau(yc)$ and $T_2=\tau(yd)$ under the edge $xy$. Note that the planting stalk $uv$ of $T$ is transformed to the red edge $xy$, the unique one at level 3 in $\phi(T)$, and the previously constructed red edges in $T_1=\phi(\tau(vc))$ and $T_2=\phi(\tau(vd))$ are transformed to red edges in $\phi(T)$ by shifting them from level 3 to level 6.
}
\end{exa}

\begin{figure}[ht]
\begin{center}
\includegraphics[width=3.4in]{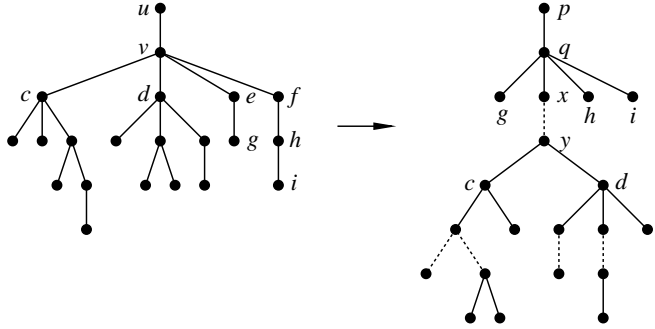}
\end{center}
\caption{\small A planted tree with and its corresponding tree.} \label{fig:example-case-3}
\end{figure}

\medskip
From the construction of $\phi$, we observe that the planting stalk of $T$ is transformed to the rightmost red edge at level 3 in $\phi(T)$, and that the other red edges recursively constructed so far (in $T_j$) are transformed to red edges in $\phi(T)$, either by shifting from level $3i$ to level $3i+3$ or by remaining at level $3i$ (as the ones in $T_{r}$ of Case 1 or in $T_{r-1}$ of Case 2), $i\ge 1$.  Hence the number of red edges in $\phi(T)$ equals the number of exterior edges in $T$.

\smallskip
\subsection{Finding $\phi^{-1}$} Indeed the map $\phi^{-1}$ can be recursively constructed by reversing the steps involved in the construction of $\phi$. To be more precise, we describe the construction below.

Given a $T\in\P_n$, if $T$ contains no red edges then $T$ is a bouquet of size $n$ and we define $\phi^{-1}(T)$ to be a path of length $n$. Otherwise, $T$ contains at least one red edge. Let $xy$ be the rightmost red edge at level 3 of $T$, and let $pqxy$ be the path from the root $p$ to $y$. Let $w_1,\dots,w_d$ be the children of $y$, for some $d$ ($d\ge 0$). The tree $\phi^{-1}(T)$ is recursively constructed with respect to $xy$ according to the following three cases.

\smallskip
 Case 1. \emph{Vertex $x$ has more than one child.} Let $Q=T-\tau(xy)$. For $1\le j\le d$, we first construct the planted subtrees $T_j=\phi^{-1}(\tau(yw_j))$ and  $T_{d+1}=\phi^{-1}(Q)$. Then $\phi^{-1}(T)$ is recovered by adding the subtrees $T_1,\dots,T_{d+1}$ under a new edge, say $uv$.  Note that  the red edge $xy$ of $T$ is replaced by the planting stalk $uv$ of $\phi^{-1}(T)$.

\smallskip
 Case 2. \emph{Vertex $x$ has only one child and there is another path $P$ of length at least 2 starting from $q$.} Since $xy$ is the rightmost red edge at level 3 of $T$, the path $P$ must be on the left of the edge $qx$.
 Note that there might be some edges, say $qz_1,\dots,qz_t$ ($t\ge 0$), on the right of $qx$. Let $Q=T-\tau(qx)-\{qz_1,\dots,qz_t\}$. For $1\le j\le d$, form the planted subtrees $T_j=\phi^{-1}(\tau(yw_j))$. Let $T_{d+1}=\phi^{-1}(Q)$ and let $T_{d+2}$ be a path of length $t+1$. Then $\phi^{-1}(T)$ is recovered by adding the subtrees $T_1,\dots,T_{d+2}$ under a new edge $uv$. Note that the planting stalk $uv$ of $\phi^{-1}(T)$ replaces the red edge $xy$ of $T$, and the path $T_{d+2}\subseteq\phi^{-1}(T)$ replaces the edges $\{qx,qz_1,\dots,qz_t\}\subseteq T$.

 \smallskip
 Case 3. \emph{Vertex $x$ has only one child and there are no other paths of length at least 2 starting from $q$.} In this case $xy$ is the unique red edge at level 3 in $T$, and there might be some edges emanating from $q$ on either side of the edge $qx$. Suppose that there are $t_1$ (resp. $t_2$) edges on the left (resp. right) of $qx$. For $1\le j\le d$, form the planted subtrees $T_j=\phi^{-1}(\tau(yw_j))$. Let $T_{d+1}$ and $T_{d+2}$ be two paths of length $t_1+1$ and $t_2+1$, respectively. Then  $\phi^{-1}(T)$ is recovered by adding the subtrees $T_1,\dots,T_{d+2}$ under a new edge $uv$.

From the construction of $\phi^{-1}$, we observe that the rightmost red edge at level 3 in $T$ is transformed to the planting stalk of $\phi^{-1}(T)$, and that the exterior edges recursively constructed so far (in $T_j$) remain exterior edges in $\phi^{-1}(T)$. Hence the number of exterior edges in $\phi^{-1}(T)$ equals the number of red edges in $T$.

\smallskip
We have established the following bijection.

\begin{pro} \label{pro:equidistribution-planted trees} For $0\le k\le n-2$, there is a bijection $\phi:\P_n\rightarrow\P_n$ such that a planted tree $T\in\P_n$ with $k$ exterior edges is carried to the corresponding planted tree $\phi(T)$ containing $k$ edges at level $h$ with $h\equiv 0$ (mod 3).
\end{pro}

Now we are able to establish the bijection $\Phi$ requested in Theorem \ref{thm:equidistribution-trees} as well as in Theorem \ref{thm:main-2}.

Given an ordered tree $T\in\T_n$ with $k$ exterior edges, let $u$ be the root of $T$ and let $v_1,\dots,v_r$ be the children of $u$, for some $r\ge 1$. Then $T$ can be decomposed into $r$ planted subtrees
\[T=\tau(uv_1)\cup\cdots\cup\tau(uv_r).\]
Suppose that $\tau(uv_i)$ contains $k_i$ exterior edges, where $k_1+\cdots+k_r=k$. Making use of the bijection $\phi$ in Proposition \ref{pro:equidistribution-planted trees}, we find the corresponding planted subtrees $T_i=\phi(\tau(uv_i))$  ($1\le i\le r$), where $T_i$ contains $k_i$ red edges. Then the corresponding tree $\Phi(T)=T_1\cup \cdots\cup T_k$, obtained by merging the roots of $T_1,\dots,T_k$, contains $k$ red edges, i.e., $k$ edges at level $h$ with $h\equiv 0$ (mod 3).
This completes the proof of Theorem  \ref{thm:equidistribution-trees}.

\smallskip
\begin{exa} \label{exa:example-bijection} {\rm
Given the Dyck path $\pi$, shown on the left of Figure \ref{fig:preorder}, with 2 blocks and 4 exterior steps, we find the corresponding ordered tree $T=\Lambda(\pi)$, shown on the right of Figure \ref{fig:preorder}, and decompose $T$ into two planted subtrees $T=T_1\cup T_2$. Following Examples \ref{exa:example-phi} and \ref{exa:example-phi-2}, we construct the trees $\phi(T_1)$ and $\phi(T_2)$, respectively. Then the corresponding tree $\Phi(T)$ is obtained by merging the roots of $\phi(T_1)$ and $\phi(T_2)$, shown on the right of Figure \ref{fig:bijection}. Note that $\Phi(T)$ contains 4 red edges. Hence, by $\Lambda^{-1}$, we obtain the corresponding Dyck path $\Pi(\pi)=\Lambda^{-1}(\Phi(\Lambda(\pi)))$, shown on the left of Figure \ref{fig:bijection}, which contains 4 up steps at height $h$ with $h\equiv 0$ (mod 3).
}
\end{exa}

\begin{figure}[ht]
\begin{center}
\includegraphics[width=4.9in]{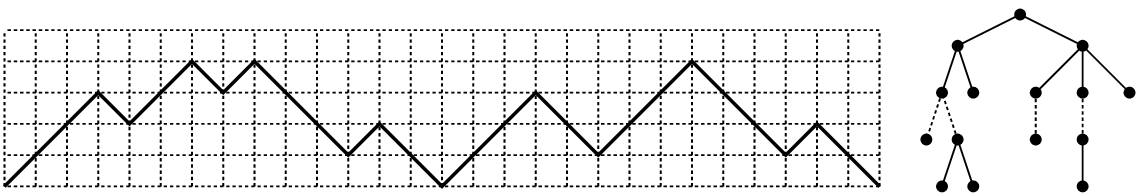}
\end{center}
\caption{\small A Dyck path and the corresponding ordered tree.} \label{fig:bijection}
\end{figure}

\section{Generating functions}
In this section, for $m\ge 2$ and $R\subseteq\{0,1,\dots,m-1\}$ ($R\neq\emptyset$), we study the generating function $G^{(R;m)}$ for Dyck paths counted according to length and number of up steps at height $h$ such that $h\equiv c$ (mod $m$) and $c\in R$.
Let $\lambda$ be a boolean function defined by $\lambda(\mbox{true})=1$ and $\lambda(\mbox{false})=0$.
By abuse of notation, let
\[ R-i=\{c':c-i+m\equiv c'\mbox{ (mod $m$), }c\in R\}.
\]

\begin{thm} \label{thm:continued} For $m\ge 2$ and a nonempty set $R\subseteq\{0,1,\dots,m-1\}$, the generating function $G^{(R;m)}$ satisfies the equation
\[
G^{(R;m)}=\frac{1}{\displaystyle 1-
          \frac{xy^{\lambda(1\in R)}}{\displaystyle 1-
          \frac{xy^{\lambda(2\in R)}}{\displaystyle
          \frac{\ddots}{\displaystyle 1-
          \frac{xy^{\lambda(m-1\in R)}}{1-xy^{\lambda(0\in R)}G^{(R;m)}}}}}}.
\]
\end{thm}

\begin{proof} For $0\le i\le m-1$, we enumerate the paths $\pi\in\C_n$ with respect to the number of up steps at height $h$ with $h\equiv c$ (mod $m$) and  $c\in R-i$. By the \emph{first-return decomposition} of Dyck paths, a non-trivial path $\pi\in\C_n$ has a factorization $\pi=\U\mu\D\nu$, where $\mu$ and $\nu$ are Dyck paths of certain lengths (possibly empty). We observe that $y$ marks the first step $\U$ if $1\in R-i$. Moreover, the other up steps in the first block  $\U\mu\D$ that satisfy the height constrain are the up steps in $\mu$ at height $h$ with $h\equiv c-1+m$ (mod $m$). Hence $G^{(R-i;m)}$ satisfies the following equation
\[G^{(R-i;m)}=1+xy^{\lambda(1\in R-i)}G^{(R-i-1;m)}G^{(R-i;m)}.
\]
Hence we have
\[
G^{(R-i;m)}=\frac{1}{1-xy^{\lambda(1\in R-i)}G^{(R-i-1;m)}}.
\]
By iterative substitution and the fact $R-m=R$, the assertion follows.
\end{proof}

\smallskip
\begin{exa} \label{exa:GFproof}
{\rm
Take $m=3$ and $R=\{0\}$, we have
\[
G^{(0;3)}=\frac{1}{\displaystyle 1-
          \frac{x}{\displaystyle 1-
          \frac{x}{1-xyG^{(0;3)}}}},
\]
which is equivalent to
\[ xy(1-x)(G^{(0;3)})^2-(1-2x+xy)G^{(0;3)}+(1-x)=0.\]
Solving this equation yields
\[G^{(0;3)} = {\displaystyle \frac{1-2x+xy-\sqrt{(1-xy)^2-4x(1-x)(1-xy)}}{2xy(1-x)}},\]
which coincides with the generating function for Dyck paths counted by length and number of exterior pairs (cf. \cite[Theorem 2.3]{DS}).
}
\end{exa}

\section{A bijective proof of Theorem \ref{thm:main}}

 Let $\A^{(m-1;m)}_{n,j}\subseteq\C_n$ (resp. $\A^{(0;m)}_{n,j}\subseteq\C_n$) be the set of paths containing exactly $j$ up steps at height $h$ with $h\equiv m-1$ (resp. $h\equiv 0$) (mod $m$).
In this section, we shall prove Theorem \ref{thm:main} by establishing the following bijection.

\begin{thm} \label{thm:Psi} For the Dyck paths in $\C_n$ of height at least $m-1$, the following results hold.
\begin{enumerate}
\item For $j\ge 2$, there is a bijection $\Psi_j$ between $\A^{(m-1;m)}_{n,j}$ and $\A^{(0;m)}_{n,j-1}$.
\item For $j=1$, there is a bijection $\Psi_1$ between $\A^{(m-1;m)}_{n,1}$ and the set $\B\subseteq\A^{(0;m)}_{n,0}$, where $\B$ consists of the paths that contain no up steps at height $h$ with $h\equiv 0$ (mod $m$) and contain at least one up step at height $h'$ with $h'\equiv m-1$ (mod $m$).
\end{enumerate}
\end{thm}

\smallskip
Fix an integer $m\ge 2$. Given a $\pi\in\C_n$ of height at least $m-1$, we cut $\pi$ into segments by lines of the form $L_i:y=mi-1$ ($i\ge 1$). The segments $\omega\subseteq\pi$ are classified into the following categories.
\begin{enumerate}
\item[(S1)] Segment $\omega$ begins with an up step starting from a line $L_i$, for some $i\ge 1$, ends with the first down step returning to the line $L_i$ afterward, and never touches the line $L_{i+1}$. We call such a segment an \emph{above-block} on $L_i$.
\item[(S2)] Segment $\omega$ begins with a down step starting from a line $L_i$, for some $i\ge 1$, ends with the first up step reaching the line $L_i$ afterward, and never touches the line $L_{i-1}$. We call such a segment an \emph{under-block} on $L_i$.
\item[(S3)] Segment $\omega$ is called an \emph{upward link} if $\omega$ begins with an up step starting from a line $L_i$, for some $i\ge 1$, and ends with the first up step reaching the line $L_{i+1}$ afterward.
\item[(S4)] Segment $\omega$ is called a \emph{downward link} if $\omega$ begins with a down step starting from a line $L_i$, for some $i\ge 2$, and ends with the first down step returning to the line $L_{i-1}$ afterward.
\item[(S5)] The segment from the origin to the first up step that reaches the line $L_1$ is called the \emph{initial segment} of $\pi$. The segment starting from the last down step that leaves the line $L_1$ to the endpoint of $\pi$ is called the \emph{terminal segment} of $\pi$.
\end{enumerate}

\smallskip
\begin{exa} \label{exa:decomposition} {\rm
Take $m=3$. The Dyck path $\pi$ shown in Figure \ref{fig:upward-block}(a) is decomposed into nine segments $\pi=\omega_1\cdots\omega_9$, where $\omega_1=[O,A]$ is the initial segment, $\omega_9=[H,I]$ is the terminal segment, $\omega_2=[A,B]$, $\omega_5=[D,E]$, and $\omega_8=[G,H]$ are above-blocks, $\omega_3=[B,C]$ and $\omega_6=[E,F]$ are under-blocks, $\omega_4=[C,D]$ is an upward link, and $\omega_7=[F,G]$ is a downward link.
}
\end{exa}

\begin{figure}[ht]
\begin{center}
\includegraphics[width=4.4in]{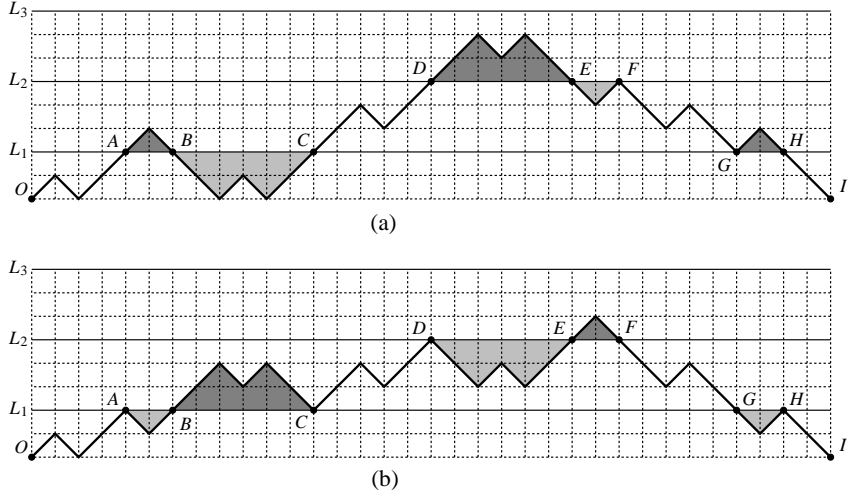}
\end{center}
\caption{\small Decomposition of a Dyck path by lines of the form $y=3i-1$ ($i\ge 1$).} \label{fig:upward-block}
\end{figure}

\smallskip
We have the following immediate observations.
\smallskip
\begin{lem} \label{lem:properties} According to the above decomposition of $\pi\in\C_n$ with respect to lines of the form $L_i:y=mi-1$ ($i\ge 1$), the following facts hold.
\begin{enumerate}
\item An above-block $\omega$ contains a unique up step (i.e., the first step of $\omega$) at height $h$ with $h\equiv 0$ (mod $m$), and contains no up steps at height $h'$ with $h'\equiv m-1$ (mod $m$).
\item An under-block $\omega$ contains a unique up step (i.e., the last step of $\omega$) at height $h$ with $h\equiv m-1$ (mod $m$), and contains no up steps at height $h'$ with $h'\equiv 0$ (mod $m$).
\item The first (resp. last) step of an upward link $\omega$ is the unique up step at height $h$ with $h\equiv 0$ (resp. with $h\equiv m-1$) (mod $m$) contained in $\omega$.
\item The last step of the initial segment of $\pi$ is the unique up step at height $m-1$ contained in $\omega$.
\item A downward link and the terminal segment of $\pi$ contain no up steps at height $h$ with $h\equiv 0$ or $m-1$ (mod $m$).
\end{enumerate}
\end{lem}

\smallskip
 For the above-blocks and under-blocks $\omega$ on some line $L_i$, we define an operation $\Gamma$ on $\omega$ by letting $\Gamma(\omega)$ be the segment obtained from $\omega$ by reflecting $\omega$ about the line $L_i$. Note that $\Gamma(\omega)$ is an under-block (resp. above-block) on $L_i$ if $\omega$ is an above-block (resp. under-block) on $L_i$. Making use of this operation, we define an involution $\Omega:\C_n\rightarrow\C_n$ as follows.

\smallskip
\noindent{\bf The involution $\Omega$.}
Given a $\pi\in\C_n$, if the height of $\pi$ is less than $m-1$, then we define $\Omega(\pi)=\pi$. Otherwise, the path $\pi$ has a factorization $\pi=\omega_1\cdots\omega_d$ ($d\ge 2$), called the \emph{standard form}, with respect to lines of the form $L_i:y=mi-1$ $(i\ge 1$), where $\omega_1$ is the initial segment, $\omega_d$ is the terminal segment, and each $\omega_r$ is a segment in one of the four categories (S1)--(S4), for $2\le r\le d-1$. The map $\Omega$ is defined by carrying $\pi$ to $\Omega(\pi)=\omega_1\widehat{\omega}_2\cdots\widehat{\omega}_{d-1}\omega_d$, where
\[\widehat{\omega}_r=\left\{ \begin{array}{ll}
                        \Gamma(\omega_r) & \mbox{if $\omega_r$ is an above-block or an under-block} \\
                        \omega_r &\mbox{if $\omega_r$ is an upward link or a downward link},
                   \end{array}
           \right.
\]
for $2\le r\le d-1$. It is obvious that $\Omega$ is an involution.

\smallskip
\begin{exa} \label{exa:involution} {\rm
Take $m=3$ and the path $\pi$ shown in Figure \ref{fig:upward-block}(a). As shown in Example \ref{exa:decomposition}, $\pi$ is factorized into the standard form $\pi=\omega_1\dots\omega_9$.  The corresponding path $\Omega(\pi)=\omega_1\widehat{\omega}_2\dots\widehat{\omega}_8\omega_9$ is shown in Figure \ref{fig:upward-block}(b), where $\widehat{\omega}_2=\Gamma(\omega_2)$, $\widehat{\omega}_3=\Gamma(\omega_3)$, $\widehat{\omega}_4=\omega_4$, $\widehat{\omega}_5=\Gamma(\omega_5)$, $\widehat{\omega}_6=\Gamma(\omega_6)$, $\widehat{\omega}_7=\omega_7$, and $\omega_8=\Gamma(\omega_8)$.
}
\end{exa}

\smallskip
Let $F^{(m)}_{n,j,k}\subseteq\C_n$ be the set of paths containing $j$ up steps at height $h$ with $h\equiv m-1$ (mod $m$) and $k$ up steps at height $h'$ with $h'\equiv 0$ (mod $m$).

\smallskip
\begin{pro} \label{pro:involution} For $j\ge 1$ and $k\ge 0$, the involution $\Omega$ induces a bijection $\Omega_{j,k}:F^{(m)}_{n,j,k}\rightarrow F^{(m)}_{n,k+1,j-1}$.
\end{pro}

\begin{proof}In particular, for $(j,k)=(1,0)$, we define $\Omega_{1,0}:F^{(m)}_{n,1,0}\rightarrow F^{(m)}_{n,1,0}$ to be an identity mapping, i.e., $\Omega_{1,0}(\pi)=\pi$, for $\pi\in F^{(m)}_{n,1,0}$.

For $(j,k)\neq (1,0)$, given a $\pi\in F^{(m)}_{n,j,k}$, we factorize $\pi$ into the standard form $\pi=\omega_1\cdots\omega_d$ ($d\ge 2$), with respect to lines $L_i:y=mi-1$ ($i\ge 1$).  Suppose that there are $t$ segments among $\omega_2,\dots,\omega_{d-1}$, which are upward links. Since $\pi$ contains $j$ up steps at height $h$ with $h\equiv m-1$ (mod $m$) and $k$ up steps at height $h'$ with $h'\equiv 0$ (mod $m$), by Lemma \ref{lem:properties}, there are $j-1-t$ segments $\mu_1,\dots,\mu_{j-1-t}\in\{\omega_2,\dots,\omega_{d-1}\}$ that are under-blocks and $k-t$ segments $\nu_1,\dots,\nu_{k-t}\in\{\omega_2,\dots,\omega_{d-1}\}$ that are above-blocks. Under the involution $\Omega$, the corresponding path $\Omega(\pi)$ contains $j-1-t$ above-blocks $\widehat{\mu}_1,\dots,\widehat{\mu}_{j-1-t}$ and $k-t$ under-blocks $\widehat{\nu}_1,\dots,\widehat{\nu}_{k-t}$. Along with the $t$ upward links in $\Omega(\pi)$ and the initial segment, by Lemma \ref{lem:properties},  $\Omega(\pi)$ contains $k+1$ up steps at height $h$ with $h\equiv m-1$ (mod $m$) and $j-1$ up steps at height $h'$ with $h'\equiv 0$ (mod $m$). Hence $\Omega_{j,k}(\pi)=\Omega(\pi)\in F^{(m)}_{n,k+1,j-1}$.

It is easy to see that $\Omega_{j,k}^{-1}=\Omega|_{F^{(m)}_{n,k+1,j-1}}=\Omega_{k+1,j-1}:F^{(m)}_{n,k+1,j-1}\rightarrow F^{(m)}_{n,j,k}$.
\end{proof}

\smallskip
\begin{exa} {\rm
Following Example \ref{exa:involution}, the path $\pi$ shown in Figure \ref{fig:upward-block}(a) contains four up steps at height $h$ with $h\equiv 2$ (mod 3) and four up steps at height $h'$ with $h'\equiv 0$ (mod 3). The corresponding path $\Omega_{4,4}(\pi)$, shown in Figure \ref{fig:upward-block}(b), contains five up steps at height $h$ with $h\equiv 2$ (mod 3) and three up steps at height $h'$ with $h'\equiv 0$ (mod 3).
}
\end{exa}

\smallskip
\noindent{\em Proof of Theorem \ref{thm:Psi}.} (i) For $j\ge 2$, we have
$\A^{(m-1;m)}_{n,j}=\cup_{k\ge 0} F^{(m)}_{n,j,k}$  and $\A^{(0;m)}_{n,j-1}=\cup_{k\ge 0} F^{(m)}_{n,k+1,j-1}$.  It follows from Proposition \ref{pro:involution} that the map $\Psi_j:\A^{(m-1;m)}_{n,j}\rightarrow\A^{(0;m)}_{n,j-1}$ is established by the refinement,  \[\Psi_j|_{F^{(m)}_{n,j,k}}=\Omega_{j,k}:F^{(m)}_{n,j,k}\rightarrow F^{(m)}_{n,k+1,j-1}, \mbox{ for $k\ge 0$}.\]

 (ii) For $j=1$, we have $\A^{(m-1;m)}_{n,1}=\cup_{k\ge 0} F^{(m)}_{n,1,k}$ and $\B=\cup_{k\ge 0} F^{(m)}_{n,k+1,0}$. It follows from Proposition \ref{pro:involution} that the map $\Psi_1:\A^{(m-1;m)}_{n,1}\rightarrow \B$ is established by the refinement,  \[\Psi_1|_{F^{(m)}_{n,1,k}}=\Omega_{1,k}:F^{(m)}_{n,1,k}\rightarrow F^{(m)}_{n,k+1,0}, \mbox{ for $k\ge 0$}.\] \qed

\smallskip
Now we are able to prove Theorem \ref{thm:main}. For $j\ge 2$, by Theorem \ref{thm:Psi}(i), we have
\[ [y^jx^n]\{G^{(m-1;m)}-y\cdot G^{(0;m)}\}=g_{n,j}^{(m-1;m)}-g_{n,j-1}^{(0;m)}=|\A^{(m-1;m)}_{n,j}|-|\A^{(0;m)}_{n,j-1}|=0.\]
For $j=1$, by Theorem \ref{thm:Psi}(ii), we have $g_{n,1}^{(m-1;m)}=|\A^{(m-1;m)}_{n,1}|=|\B|$, where $\B$ consists of the paths in $\C_n$ that contain no up steps at height $h$ with $h\equiv 0$ (mod $m$) and contain at least one up step at height $h'$ with $h'\equiv m-1$ (mod $m$). Hence
\[ [y^1x^n]\{G^{(m-1;m)}-y\cdot G^{(0;m)}\}=g_{n,1}^{(m-1;m)}-g_{n,0}^{(0;m)}=|\B|-|\A^{(0;m)}_{n,0}|
=-\frac{U_{m-2}(\frac{1}{2\sqrt{x}})}{\sqrt{x}U_{m-1}(\frac{1}{2\sqrt{x}})},\]
which is the negative of the number of paths in $\C_n$ of height at most $m-2$.
Moreover, $[y^0x^n]\{G^{(m-1;m)}-y\cdot G^{(0;m)}\}=g_{n,0}^{(m-1;m)}$ is also the number of paths in $\C_n$ of height at most $m-2$. This completes the proof of Theorem \ref{thm:main}.

\section{Concluding Notes}
Given a positive integer $s$, an \emph{$s$-ary} path of length $n$ is a lattice path from $(0,0)$ to $((s+1)n,0)$, using up step $(1,1)$ and grand down step $(1,-s)$,  that never passes below the $x$-axis. When $s=1$ it is an ordinary Dyck path. One can consider the $s$-generalization of pyramids and exterior pairs on $s$-ary paths. For example, a pyramid of height $k$ is a succession of $sk$ up steps followed immediately by $k$ down steps. An \emph{exterior down step} is a down step that does not belong to any pyramid. Let $p^{(s)}_{n,k}$ (resp. $e^{(s)}_{n,k}$) be the number of $s$-ary paths of length $n$ with a pyramid weight of $k$ (resp. with $k$ exterior down steps), and let $P$ and $E$ be the generating functions for $p^{(s)}_{n,k}$ and $e^{(s)}_{n,k}$, respectively, where
\[ P=P(x,y)=\sum_{n\ge 0} \sum_{k\ge 0} p^{(s)}_{n,k} y^kx^n, \qquad E=E(x,y)=\sum_{n\ge 0}\sum_{k\ge 0} e^{(s)}_{n,k} y^kx^n.
\]
Note that $E(x,y)=P(xy,y^{-1})$ since an $s$-ary path of length $n$ with a pyramid weight of $k$ contains $n-k$ exterior down steps.

\smallskip
\begin{pro} \label{pro:s-exterior} The generating functions $P$ and $E$ satisfy respectively the equations
\[P=1+x(P^s-\frac{1-y}{1-xy})P,\qquad E=1+x(yE^s+\frac{1-y}{1-x})E.\]
\end{pro}

\begin{proof} By the first-return decomposition of $s$-paths, a nontrivial $s$-path $\pi$ has a factorization $\pi=U_1\mu_1\cdots U_s\mu_s D\nu$, where $D$ is the first (grand) down step that returns to the $x$-axis, $U_i$ is the last up step in the first block $\beta=U_1\mu_1\cdots U_s\mu_s D\subseteq\pi$, which rises from the line $y=i-1$ to the line $y=i$ ($1\le i\le s$), and $\mu_1,\dots,\mu_s,\nu$ are $s$-ary paths of certain lengths (possibly empty). To enumerate the $s$-ary paths with respect to pyramid weight and length, we observe that the first down step $D$ is marked $y$ if and only if the first block $\beta$ is a pyramid, in which case $\mu_1=\cdots=\mu_{s-1}=\emptyset$ and $\mu_s$ is a pyramid of certain length. Hence $P$ satisfies the equation
\[P=1+x(P^s-\frac{1}{1-xy}+\frac{y}{1-xy})P.\]
Similarly, if we enumerate the $s$-ary paths with respect to the number of exterior down steps and length, then the first down step $D$ is marked $y$ if and only if the first block $\beta$ is not a pyramid. Hence $E$ satisfies the equation
\[E=1+x(y(E^s-\frac{1}{1-x})+\frac{1}{1-x})E,\]
as required.
\end{proof}

We are interested to know if there is any statistic regarding up steps, which is equidistributed with $p^{(s)}_{n,k}$ (or $e^{(s)}_{n,k}$) on the $s$-ary paths.

Theorem \ref{thm:main} gives a relation between the two generating functions $G^{(m-1;m)}$ and $G^{(0,m)}$.
It is natural to consider if there is any relation between
$G^{(i;m)}$ and $G^{(j,m)}$, for $0\le i,j\le m-1$. In fact, we have two promising observations from some evidence generated by computer.
We are interested in an algebraic or combinatorial proof.

\smallskip
\begin{con} The following relations hold.
\begin{enumerate}
\item For $m\ge 4$, we have
\[G^{(m-2,m)}-G^{(1,m)}=\frac{(1-t)U_{m-4}(\frac{1}{2\sqrt{x}})}{U_{m-2}(\frac{1}{2\sqrt{x}})-y\sqrt{x}U_{m-3}(\frac{1}{2\sqrt{x}})}.\]

\item For $m\ge 6$, we have
\[G^{(m-3,m)}-G^{(2,m)}=\frac{(1-t)U_{m-6}(\frac{1}{2\sqrt{x}})}{U_{m-2}(\frac{1}{2\sqrt{x}})-yU_{m-4}(\frac{1}{2\sqrt{x}})+\sqrt{x}U_{m-5}(\frac{1}{2\sqrt{x}})}.\]
\end{enumerate}
\end{con}

\end{document}